\documentclass[11pt]{article}

\usepackage{latexsym}
\usepackage{amssymb}
\usepackage{amsthm}
\usepackage{amscd}

\usepackage{amsmath}

\usepackage{mathrsfs}
\usepackage[all]{xy}
\usepackage{hyperref} 
\usepackage[usenames,dvipsnames]{color}
\usepackage{graphicx,eepic}
\usepackage{float}

\usepackage{color}

\theoremstyle{definition}
\newtheorem* {theorem*}{Theorem}
\newtheorem{theorem}{Theorem}[section]

\theoremstyle{definition}

\newtheorem* {example*}{Example}

\newtheorem{lemma}[theorem]{Lemma}
\theoremstyle{definition}

\theoremstyle{definition}
\newtheorem* {notation}{Notation}

\newtheorem{proposition}[theorem]{Proposition}
\newtheorem{corollary}[theorem]{Corollary}

\newtheorem* {remark}{Remark}
\theoremstyle{definition}
\newtheorem {example}[theorem]{Example}
\theoremstyle{definition}

\theoremstyle{definition}

\theoremstyle{definition}

\xyoption{dvips}

\usepackage{fullpage}

\numberwithin{equation}{section}

\def\modu{\ (\mathrm{mod}\ }

\def\({\left(}
\def\){\right)}
       \newcommand{\QQ}{\mathbb{Q}}

\newcommand{\cD}{\mathcal{D}}

\def\NN{\mathbb{N}}
    \def\ZZ{\mathbb{Z}} \def\Aut{\mathrm{Aut}}

\DeclareSymbolFont{bbold}{U}{bbold}{m}{n}
\DeclareSymbolFontAlphabet{\mathbbold}{bbold}
\def\One{\mathbbold{1}}

\def\barr{\begin{array}}
\def\earr{\end{array}}
\def\ba{\begin{aligned}}
\def\ea{\end{aligned}}
\def\be{\begin{equation}}
\def\ee{\end{equation}}

\def\qquand{\qquad\text{and}\qquad}

\def\I{\mathbf{I}}

\def\omdef{\overset{\mathrm{def}}}

\def\ds{\displaystyle}

\newcommand{\id}{\operatorname{id}}

\def\ben{\begin{enumerate}}
\def\een{\end{enumerate}}

\def\inv{\mathrm{Inv}}

\def\Imin{\Omega}
\newcommand{\ellinv}[1]{\ell^{#1}}
\def\ellprime{\ell' }

\makeatletter
\renewcommand{\@makefnmark}{\mbox{\textsuperscript{}}}
\makeatother

\allowdisplaybreaks[1]

\UseCrayolaColors

\begin{document}
\title{Variations of the Poincar\'e series for affine Weyl groups and $q$-analogues of Chebyshev polynomials}

\author{
   Eric Marberg\footnote{This author was supported through a fellowship from the National Science Foundation.}
\\
Department of Mathematics \\
Stanford University \\
{\tt emarberg@stanford.edu}
    \and
   Graham White
   \\
Department of Mathematics \\
Stanford University \\
{\tt grwhite@math.stanford.edu}
}

\date{}

\maketitle

\begin{abstract}
Let $(W,S)$ be a Coxeter system and write $P_W(q)$ for its Poincar\'e series.  Lusztig has shown that  the quotient $P_W(q^2)/P_W(q)$  is equal to a certain power series  $L_{W}(q)$,   defined by specializing one variable in the    generating function recording the lengths and  absolute lengths of the involutions in $W$.
The simplest inductive method of proving this result for finite Coxeter groups  suggests 
a  natural bivariate generalization $L^J_W(s,q) \in \ZZ[[s,q]]$  depending on a subset  $J\subset S$.
This new power series specializes to $L_W(q)$ when $s=-1$ and is given explicitly by a   sum of rational functions over the involutions which are minimal length representatives of the double cosets  of the  parabolic subgroup $W_J$ in  $W$.
When $ W$ is an affine Weyl group, we consider the renormalized power series  $T_{ W}(s,q) = L^J_W(s,q) / L_W(q)$ with $J$ given by the generating set of the corresponding finite Weyl group.
We show that when  $W$ is an affine Weyl group of type $A$, the power series $T_W(s,q)$  is actually a polynomial in $s$ and $q$ with nonnegative coefficients, which turns out to be a $q$-analogue recently studied by Cigler of the Chebyshev polynomials of the first kind, arising in a completely different context. 
\end{abstract}

\setcounter{tocdepth}{2}
\tableofcontents

\section{Introduction}

\subsection{Background and motivation}\label{sect1} 
 
Let $(W,S)$ be a Coxeter system with length function $\ell : W \to \NN$.  The \emph{Poincar\'e series} of $(W,S)$ is  the formal power series (in an indeterminate $q$) given by
\[ P_W(q)  = \sum_{w \in W} q^{\ell(w)} \in \ZZ[[q]].\]
This power series is well-defined if and only if the rank  of $(W,S)$ is finite, and in this work, therefore, we require all Coxeter systems  $(W,S)$ to have $|S| < \infty$.
If $W$ is finite then $P_W(q)$ is obviously a polynomial, and in general  $P_W(q)$ is always a rational power series; see 
 \cite{Hu,Steinberg}.

Lusztig \cite{LV2,L6} has introduced an interesting analogue of the Poincar\'e series defined in terms of the twisted involutions in a Coxeter group, and our main object of study here
is a natural bivariate generalization of this power series. To motivate its definition, we review some relevant  information from \cite{LV2,L6}.

To begin, 
let $\Aut(W,S)$ denote the group of automorphisms of $W$ preserving $S$, and fix an involution (that is, a self-inverse automorphism) $* \in \Aut(W,S)$. We denote the action of $*$ on  elements $w \in W$ by $w^*$,
and write
\[ \I_* = \I_*(W) \omdef= \{ w \in W : w^{-1} = w^*\}\]
for the corresponding set of \emph{twisted involutions} in $W$.
The ``twisted'' analogue of $P_W(q)$ is  the formal power series
\be
\label{Ldef} L_{W,*}(q) = \sum_{w \in \I_*} q^{\ell(w)} \(\tfrac{q-1}{q+1}\)^{\ellinv{*}(w)} \in \ZZ[[q]]
\ee
where on the right  side 
 $\ellinv{*}$ 
 denotes the \emph{twisted absolute length function} defined by Hultman in \cite{H1}, which is characterized explicitly as the unique map $\I_* \to \NN$
 such that
 \ben
\item[(a)] $\ellinv{*}(1) = 0$;
\item[(b)] $\ellinv{*}$ is constant on $*$-twisted conjugacy classes, i.e., $\ellinv{*}(sws^*) = \ellinv{*}(w)$ for all $s \in S$;
\item[(c)] $\ellinv{*}(ws) - \ellinv{*}(w) = \ell(ws)-\ell(w)$ whenever $s \in S$ and $w \in \I_*$ are such that $ws \in \I_*$.
\een
Note  in (c) that $ws \in \I_*$ if and only if $ws=s^*w$. The function $\ellinv{*}$ 
is the same as the map denoted   $\phi$  in \cite{LV2,L6}.
Lusztig's paper \cite[\S5.8]{LV2}  appears to be the first place in the literature where the power series \eqref{Ldef} is considered, and for this reason we denote it by the letter $L$.

\begin{notation}
When $*$ is the identity automorphism (that is, $*=\id$) we abbreviate by setting 
\[
\I(W) = \I_{\id}(W) 
\qquand
\ellprime = \ellinv{\id}
\qquand
L_W(q) = L_{W,\id}(q).
\]
We will repeat this convention with a few subsequent definitions. 
 \end{notation}
 
The map $\ellprime : \I(W) \to \NN$   is  often called the \emph{absolute length function} of $W$. The value of $\ellprime$ at $w$ gives
 the minimum number of reflections whose product is $w$, and is also (when $w^2=1$)  the dimension of the $-1$-eigenspace of $w$ in the geometric representation of $(W,S)$; see \cite{Dyer}. 
 In the case when $W$ is a classical Weyl group, Incitti \cite{Incitti1,Incitti2} has derived explicit  formulas for $\ellprime$.

\begin{example}
The symmetric group $S_3 =\langle s_1,s_2\rangle$ is a Coxeter group with simple generators given by the transpositions $s_1 = (1,2)$ and $s_2=(2,3)$, and contains four involutions given by $1$, $s_1$, $s_2$, and $s_1s_2s_1 = s_2s_1s_2$, on which  $\ell$ and $\ellprime$ have values $0,1,1,3$ and $0,1,1,1$. Thus
\[
L_{S_3}(q) = 1 +2 q\cdot \tfrac{q-1}{q+1}  + q^3\cdot \tfrac{q-1}{q+1} = \tfrac{(1+q^2)(1-q+q^2)}{(1+q)}.
\]
\end{example}

\begin{example}
The  product $W\times W$ is a Coxeter group
relative to the generators $ (S\times \{1\}) \cup (\{ 1\} \times S)$.
For the automorphism $\tau \in \Aut(W\times W)$ with $(x,y) \mapsto (y,x)$, it holds that
 \[L_{W\times W,\tau}(q) = P_W(q^2)\]
since
$\I_\tau(W\times W) = \{ (w,w^{-1}) : w \in W\}$
is the $\tau$-twisted conjugacy class of $(1,1)\in W\times W$.
\end{example}

There is a common generalization of the formulas in these examples.
Define $F_{W,*}(q)$ as the length generating function of the set of fixed points of $*$ in $W$, so that
\[ F_{W,*}(q) = \sum_{\substack{w \in W\\ w=w^*}}q^{\ell(w)} \in \ZZ[[q]].\]
This power series may be realized as a special case of a multivariate generalization of the Poincar\'e series introduced by MacDonald, and so is always rational; see \cite[\S1.2]{MacDonald}.
Lusztig proves the following in \cite{L6}, which shows that $L_{W,*}(q)$  is   also rational:

\begin{theorem}[Lusztig \cite{L6}] \label{L6-thm} If $W$ is any Coxeter group then 
$L_{W,*}(q) = P_{W}(q^2) / F_{W,*}(q)$. 
\end{theorem}

\begin{remark} The results in this paper provide an independent, self-contained proof of this theorem in the special case when $W$ is an affine Weyl group of type $A$ and $*=\id$.
\end{remark}

This theorem leads to  explicit formulas for $L_W(q)$ when $W$ is a finite or affine Weyl group, once we recall the well-known factorization of the Poincar\'e series in these cases.
Assume that $W$ is finite with rank $n=|S|$.  If $V$ is the geometric representation of $W$ (see, e.g.,  \cite[\S2.4]{Carter}) then the ring of $W$-invariants in the polynomial algebra $\mathrm{Sym}(V^*)$ is itself a polynomial ring, which is minimally generated by $n$ homogeneous polynomials, whose degrees are uniquely determined up to permutation. 
These numbers are the \emph{degrees} of the basic polynomial invariants of $W,$ which we denote  as $d_1,d_2,\dots,d_n \in \NN$. If $(W,S)$ is a Coxeter system of type $A_n$ (so that $W=S_{n+1}$), then $d_i = i+1$.
The following statement was first established for Weyl groups by Chevalley, and later given a uniform proof for all finite Coxeter groups by Solomon.

\begin{theorem}[Chevalley \cite{Chevalley3}, Solomon \cite{Solomon3,Solomon4}] \label{Pfactor-thm}
If $W$ is finite then $P_W(q) = \prod_{i=1}^n \frac{1-q^{d_i}}{1-q}$.
\end{theorem}

If $W$ is a finite Weyl group, then 
there is an associated \emph{affine Weyl group}  $\tilde W$,
which is a semidirect product of $W$ with the group of translations corresponding to the coroot lattice in the geometric representation $V$ of $W$; see  \cite[\S4.2]{Hu} for the precise definition.
 When $(W,S)$ is the Coxeter system of an irreducible Weyl group,  a Coxeter system  for the associated affine Weyl group is given by $(\tilde W ,\tilde S)$, where $\tilde S  = S \cup \{ s_0\}$ with $s_0$ denoting the affine reflection in $V$ through a hyperplane normal to the highest root for $W$.
  There is a factorization of $P_{\tilde W}(q)$ analogous to Theorem \ref{Pfactor-thm}, due to Bott \cite{Bott} (see also \cite{Eriksson,Reiner}). 

\begin{theorem}[Bott \cite{Bott}] \label{bott-thm} If $W$ is a finite Weyl group then $ P_{\tilde W}(q) =  \prod_{i=1}^n \frac{1-q^{d_i}}{(1-q)(1-q^{d_i-1})}$.
\end{theorem}

These classical formulas imply the following corollaries of Theorem \ref{L6-thm}. 
It becomes clear that there is some interesting analogy between $P_W(q)$ and $L_{W,*}(q)$ on seeing these similar factorizations.

\begin{corollary}
\label{+cor} 
If $W$ is finite then $L_W(q) =\prod_{i=1}^n \frac{1+q^{d_i}}{1+q}$.
\end{corollary}

\begin{corollary} \label{Wtilde-cor}
If $W$ is a finite Weyl group then $ L_{\tilde W}(q) =   \prod_{i=1}^n \frac{1+q^{d_i}}{(1+q)(1+q^{d_i-1})}$.
\end{corollary}

Lusztig proves the first of these corollaries as \cite[Proposition 5.9]{LV2} by a direct inductive argument. This method, when adapted to  infinite Coxeter groups, suggests the definition of  a bivariate analogue of $L_{W,*}(q)$ 
which turns out to be an interesting object on its own. 
To introduce this, 
let   $J\subset S$ be any subset, write  $W_J = \langle J\rangle$ for the corresponding standard parabolic subgroup of $W$, and let 
$W^J = \{ w \in W : \ell(ws)> \ell(w) \text{ for all }s \in J\}$
denote the associated set of minimal left coset representatives. If $K \subset  J$ then we write $W_J^K $ for $(W_J)^K$.

Now, writing $P_J(q)$, $F_{J,*}(q)$, and $L_{J,*}(q)$ in place of $P_{W_J}(q)$, $F_{W_J,*}(q)$, and $L_{W_J,*}(q)$ to avoid excessive subscripts,
we define 
\be\label{LJ-def} L^J_{W,*}(s,q) = \sum_{ (w,\diamond,K)  } \(q^{\ell(w)} \cdot \(s \cdot \tfrac{1-q}{1+q}\)^{\ellinv{*}(w)} \cdot \tfrac{P_{J}(q^2)}{P_{K}(q^2)} \cdot L_{K,\diamond}(q)\) \in \ZZ[[s,q]] \ee
where the sum is over all triples $(w,\diamond,K) \in W \times \Aut(W) \times 2^S$ 
with
\be\label{cosetdata} w \in \I_* \cap W^J \qquand \diamond : x \mapsto  w\cdot x^*\cdot w^{-1} \qquand K = J \cap J^\diamond.\ee
We say that a triple $(w,\diamond,K)$ of this form is  a \emph{$(W,J,*)$-double coset datum}. 
Such data are obviously in bijection with $\I_* \cap W^J$, as well as with the set of  $(W_J,W_{J^*})$-double cosets in $W$  whose intersection with $\I_*$ is nontrivial.

As usual we write $L^J_W(s,q) = L^J_{W,\id}(s,q)$. We comment that it follows from \cite[Proposition 3.3(a)]{S} that the image of $\ellinv{*} : \I_*\to \NN$ is always finite, and so $L^J_{W,*}(s,q)$ is actually a polynomial in $s$ with coefficients in $\ZZ[[q]]$.
This power series is a generalization of $L_{W,*}(q)$ in the sense of the following lemma, which can be extracted from the discussion in \cite[\S5.8]{LV2}.

\begin{lemma}[Lusztig \cite{LV2}] \label{induct-lem} 
For any subset $J\subset S$ it   holds that
$L^J_{W,*}(-1,q) = L_{W,*}(q)$.
\end{lemma}

We include a sketch of the proof for completeness. 

\begin{proof}
If $\cD \subset W$ is a $(W_J,W_{J^*})$-double coset,
then 
 $\cD$ contains a unique element $w_\cD$ of minimal length, and  $\I_*\cap \cD  \neq \varnothing$ if and only if $w_\cD \in \I_* \cap W^J$; see  \cite[\S2.1]{GP}. In this case, if $\diamond$ and $K$ are  as in \eqref{cosetdata} for $w=w_\cD$, then it follows from \cite[Proposition 2.1.1 and Lemma 2.1.9]{GP} that
(1) for each $y \in \I_*\cap \cD $  there is a unique pair $(u,z) \in W_J^K \times \I_\diamond(W_K) $ with 
 $ y= u \cdot z \cdot w_\cD \cdot (u^*)^{-1}$, and  (2)  the map $y\mapsto  (u,z)$ is  a bijection $ \I_* \cap \cD \to W_J^K \times \I_\diamond(W_K)$ satisfying
\[ \ell(y) = 2\ell(u) + \ell(z) + \ell(w_\cD) \qquand \ellinv{*}(y) =  \ellinv{\diamond}(z) + \ellinv{*}(w_\cD).\]
From these facts, it is straightforward to deduce that 
\[
  \sum_{w \in \cD \cap \I_*} q^{\ell(w)} \(\tfrac{q-1}{q+1}\)^{\ellinv{*}(w)}
  =
q^{\ell(w_\cD)}\cdot   \(\tfrac{q-1}{q+1}\)^{\ellinv{*}(w_\cD)}
\cdot   \tfrac{P_{J}(q^2)}{P_{K}(q^2)} \cdot  L_{K,\diamond}(q)
\]
and summing this formula over all double cosets $\cD$ intersecting $\I_*$ gives the desired identity.
\end{proof}

If $W$ is finite then the sum defining $L^J_{W,*}(s,q)$  often has only a few terms, and so setting $s=-1$ translates the definition \eqref{LJ-def} to a simple recurrence for $L_{W,*}(q)$ by Lemma \ref{induct-lem}.  Such finite order recurrences allow one to check Corollary \ref{+cor} in a case-by-case fashion using induction on rank, which was Lusztig's original strategy for proving this result in \cite{LV2}.

\begin{example}
Let $W=S_{n+1}$, viewed as a Coxeter group relative to $S = \{s_1,\dots,s_{n}\}$ where $s_i = (i,i+1)$, and set   $J = \{s_1,\dots,s_{n-1}\}$. Then
$ \I(W)\cap W^J  = \{ 1 ,s_{n}\}$ so the right side of \eqref{LJ-def} has two terms, which one can check are given by
\[ L^{J}_{S_{n+1}}(s,q) = L_{S_{n}}(q) + s \cdot q \cdot \tfrac{1-q^{n}}{1+q}\cdot \tfrac{1+q^{n}}{1+q} \cdot L_{S_{n-1}}(q).\]
The recurrence given by   setting $s=-1$ implies   $L_{S_n}(q) = \prod_{i=1}^n \frac{1+q^i}{1+q}$, as predicted by Corollary \ref{+cor}.
\end{example}

In $W$ is an infinite Coxeter group and $W^J$ is also infinite (as occurs, for example, if $W_J$ is finite) then the sum defining  $L^J_{W,*}(s,q)$ has an infinite number of terms  and so is more interesting to consider, as well as more difficult to compute. Our results about this power series become nicer if we renormalize it to specialize to the identity when $s=-1$, and  we therefore define
\be\label{TJ-def} T^J_{W,*}(s,q) =  L^J_{W,*}(s,q) / L_{W,*}(q) .
\ee
When $(W,S)$ is a finite Weyl group and $(\tilde W ,\tilde S)$ is the Coxeter system of the corresponding affine Weyl group, 
there is a particularly  natural choice of the set $J$, namely, the generating set $S \subset \tilde S$ of the finite subgroup $W$. We abbreviate in this special case by setting
\[ T_{\tilde W}(s,q) =   T^S_{\tilde W,\id}(s,q).
\]
More explicitly, we have
\be
\label{Tdef} T_{\tilde W}(s,q) =
   \tfrac{ P_{\tilde W}(q)} { P_{\tilde W}(q^2)}
 \cdot  \sum_{(w,\diamond,K)} \( q^{\ell(w)} \cdot \( s \cdot \tfrac{1-q}{1+q}\)^{\ellprime (w)} \cdot \tfrac{P_W(q^2)}{F_{K,\diamond}(q)}\)
\ee
where the sum is over all triples $(w,\diamond,K)$ such that $w$ is an involution which is a minimal length left (equivalently, right) coset representative of $W$ in $\tilde W$,
 $\diamond  \in \Aut(\tilde W)$ is the inner automorphism $x\mapsto wxw$, and $K = S \cap wSw  $. With this definition, we may now discuss our results.

\subsection{Outline of results}

We write $S_n$ and $\tilde S_n$ for the Coxeter groups of rank $n-1$ and $n$ 
with the respective Coxeter diagrams
\be\label{diagrams}\barr{c} 
 \xymatrix  @=18pt{ \\
 {s_1}  \ar @{-} [r]   &  {s_2}  \ar @{-} [r]   &  {\cdots} \ar @{-} [r]     &  {s_{n-1}}    
 } 
 \\[-10pt]\\
 S_n\text{ (for $n\geq 1$)} \earr
 \qquad\qquad
\barr{c}
  \xymatrix  @=6pt{ 
 & && {s_0}\ar @{-} [dlll] \ar @{-} [drrr]
 \\
 {s_1}  \ar @{-} [rr]&   &  {s_2}  \ar @{-} [rr]  & &  {\cdots} \ar @{-} [rr] &        &  {s_{n-1}}  
 }
 \\[-10pt]\\
 \tilde S_n\text{ (for $n\geq 3$)}
 \earr
\ee
In the  degenerate cases when $n \in\{1, 2\}$, we define $\tilde S_2$ to be the infinite dihedral group and set $\tilde S_1 = S_1 = \{1\}$.
Then $S_n$ is  the Weyl group of type $A_{n-1}$ for all positive integers $n$, and $\tilde S_n$ is the corresponding affine Weyl group.
The group $S_n$ is isomorphic to the symmetric group of permutations of $\{1,2,\dots,n\}$, while $\tilde S_n$ is isomorphic to the \emph{affine symmetric group} whose definition we will review in Section \ref{affSn-sect}. 

Our main result is  an explicit formula for $T_{\tilde S_n}(s,q)$.
In stating this, we adopt the usual notational conventions 
 for $q$-factorials, $q$-binomial coefficients, and $q$-Pochhammer symbols; that is, we set
\[
[n]_q  =\tfrac{1-q^n}{1-q}
\qquand  [n]_q! = \prod_{1\leq k \leq n} [k]_q
\qquand  \tbinom{n}{k}_q = \tfrac{[n]_q!}{[k]_q! [n-k]_q!}\text{ for $0\leq k\leq n$}
\]
and define 
$(a;q)_n = (1-a)(1-aq)(1-aq^2)\cdots (1-aq^{n-1})$, with $(a;q)_0=1$.
The proof of the following theorem occupies most of Section \ref{results-sect}.

\begin{theorem}\label{main-thm}
For all $n > 0$ it holds that 
\[ T_{\tilde S_n}(s,q) = \ds \sum_{k=0}^{ \lfloor n/2\rfloor} \( s^k \cdot  q^{k^2} \cdot \tfrac{[n]_q}{[2n-2k]_q} \cdot \tbinom{n-k}{k}_q \cdot
 (-q^{k+1};q)_{n-2k}\)
 .  \]
\end{theorem}

This result will imply a number of remarkable properties of the power series $T_{\tilde S_n}(s,q)$.
For small values of $n$, the theorem gives the following formulas:
\[
\ba
T_{\tilde S_1}(s,q) &= 1
\\
T_{\tilde S_2}(s,q) &= (1+q)+qs
\\
T_{\tilde S_3}(s,q) &= (1+q +q^2 + q^3) + (q+q^2+q^3)s 
\\
T_{\tilde S_4}(s,q) &= 
(1 + q + q^2 + 2q^3 + q^4 + q^5 + q^6) + 
(q + q^2 + 2q^3 + 2q^4 + q^5 + q^6)s + q^4 s^2.
%=(1+q)(1+q^2)(1+q^3) + q(1+q)(1+q^2)^2)s + q^4 s^2
\ea
\]
Setting $q=1$ here leads to the following surprising observation. Write $T_n(x)$ to denote the $n^{\mathrm{th}}$ Chebyshev polynomial of the first kind, which we recall 
is
 the unique polynomial over $\ZZ$ satisfying $T_n(\cos( x)) = \cos(nx)$ for all $x \in \mathbb{R}$.
 
\begin{corollary}\label{cor1}
If $t_{n,k}(q) \in \QQ(q)$ are   rational functions such that $T_{\tilde S_n}(s,q) = \sum_{k\geq 0} t_{n,k}(q) (-s)^k$,
then $T_n(x)=\sum_{k\geq 0}   t_{n,k}(1) x^{n-2k}$.
\end{corollary}

\begin{proof}
The corollary follows by comparing the right side of Theorem \ref{main-thm} with $q=1$ with a well-known explicit sum formula for $T_n(x)$,
available for example as \cite[Eq.\ (14)]{Weisstein}.
\end{proof}

Thus, in a certain sense the power series $T_{\tilde S_n}(s,q)$ are $q$-analogues of the Chebyshev polynomials of the first kind. 
Amazingly, these $q$-analogues are essentially the same as ones which have appeared in recent work of Cigler in an entirely unrelated context.

In the preprints
 \cite{Cigler1,Cigler2,Cigler3}, Cigler defines and studies certain trivariate polynomials $T_n(x,s,q)$ and $U_n(x,s,q)$ 
 which he calls  the \emph{$q$-Chebyshev polynomials} of the first and second kind. These polynomials reduce to the classical Chebyschev polynomials $T_n(x)$ and $U_n(x)$ of the first and second kind on setting $s=-1$ and $q=1$, and exhibit many $q$-analogues of well-known identities for those much-studied orthogonal polynomials. 
Cigler notes  that his $q$-Chebyshev polynomials of the second kind $U_n(x,s,q)$  
are closely related to the \emph{Al-Salam and Ismail polynomials} previously studied in 
\cite{AI,IPS}; see \cite[Eq.\ (5.3)]{Cigler1}. 
The polynomials  $T_n(x,s,q)$ and $U_n(x,s,q)$ are also connected to the  bivariate $q$-analogues $p_n^{(T)}(x|q)$ and $p_n^{(U)}(x|q)$ studied by Atakishiyeva and  Atakishiyev in \cite{AA}; explicitly, the latter polynomials are obtained by  rescaling $T_n(x,-1,q)$ and $U_n(x,-q^{-1},q)$.
Finally, as Koornwinder observes in \cite[\S14.5]{Koornwinder}, all of these polynomials can be expressed as special cases of  \emph{big $q$-Jacobi polynomials}, about which we will  say more  following Corollary \ref{main2-cor}.

It is Cigler's $q$-Chebyshev polynomials of the first kind  which connect most directly to our  discussion, and so we only provide their definition.
Following \cite[Definition 2.1]{Cigler2},  we let $T_n(x,s,q)$ denote the unique element of $ \ZZ[x,s,q]$ 
 satisfying the recurrence  
\[  T_n(x,s,q) = (1+q^{n-1})  \cdot x \cdot  T_{n-1}(x,s,q) + q^{n-1}\cdot  s\cdot   T_{n-2}(x,s,q)
\qquad\text{(for $n\geq 2$)}\]
with the initial conditions $T_0(x,s,q)= 1$ and $T_1(x,s,q) = x$.
Comparing Theorem \ref{main-thm} with \cite[Theorem 2.5]{Cigler2} implies the following:

\begin{theorem}\label{main2-thm}
For all $n> 0$ it holds that $T_{\tilde S_n}(s,q) = T_n(1,s,q)$. 
\end{theorem}

\begin{remark}
From this result, we see that $T_{\tilde S_n}(s,q)$ may be defined by a  simple second-order recurrence. %, from which it follows that $T_{\tilde S_n}(-1,q)  = 1$. 
We do not know of any elementary way of proving this recurrence directly from \eqref{Tdef}. 
\end{remark}
%
%\begin{proof}
%The result follows by inspection on comparing Theorem \ref{main-thm} with \cite[Theorem 2.5]{Cigler2}. Alternatively, for a self-contained proof, it is an elementary though not entirely trivial exercise to check directly that our formula for $T_{\tilde S_n}(s,q)$ satisfies the  recurrence  defining $T_n(x,s,q)$ when $x=1$.
%\end{proof}

The polynomials $x^n \cdot T_n(1,s/x^2,q)$ satisfy the same recurrence as $T_n(x,s,q)$, and so,
 conversely, Cigler's $q$-Chebyshev polynomials are actually determined by $T_{\tilde S_n}(s,q)$:

\begin{corollary}\label{main2-cor}
For all $n> 0$ it holds that 
$  T_{n}(x,s,q) = x^n \cdot T_{\tilde S_n}( s/x^2,q)$.
\end{corollary}

 Koornwinder \cite[\S14.5]{Koornwinder} notes that the $q$-Chebyshev polynomials $T_n(x,s,q)$ satisfy 
\be\label{k} 
T_n(x,s,q) = (-s)^{\frac{1}{2}n}\cdot P_n\((-qs)^{-\frac{1}{2}}x;q^{-\frac{1}{2}},q^{-\frac{1}{2}},1,1;q\)
\ee
where $P_n(x;a,b,c,d;q) $ denotes the $n^{\mathrm{th}}$ \emph{big $q$-Jacobi polynomial}, which may be defined in terms of $q$-hypergeometric functions (see \cite[Eq.\ (93)]{Koornwinder}) as  
\[P_n(x;a,b,c,d;q) \omdef= {_3\phi_2}\( {\barr{c} q^{-n},\ q^{n+1}ab,\ qac^{-1}x  \\  qa,\ -qac^{-1}d \earr};q,q\).
\]
%where ${_r\phi_s}\({\barr{c} \alpha_1,\ \dots,\ \alpha_r \\ \beta_1\ \dots\ \beta_s \earr};q,z\)$ denotes the $q$-hypergeometric function. 
By the theorem,  $T_{\tilde S_n}(s,q)$ 
 therefore has this connection to the  big $q$-Jacobi polynomials:
\begin{corollary}
For all $n>0$ it holds that 
\[ (-s)^{-\frac{1}{2}n} \cdot T_{\tilde S_n}(s,q) = P_n\((-qs)^{-\frac{1}{2}};q^{-\frac{1}{2}},q^{-\frac{1}{2}},1,1;q\)
=
{_3\phi_2}
\( {\barr{c} q^{-n},\ q^{n},\ (-s)^{-1/2} \\  q^{1/2},\ -q^{1/2} \earr};q,q\).
\]
\end{corollary}

Surprisingly, despite the close connection between  $T_n(x,s,q)$ and $T_{\tilde S_n}(s,q)$, 
Cigler's motivations for defining the $q$-Chebyshev polynomials in \cite{Cigler1,Cigler2,Cigler3}, as well as the analogous motivations in the antecedent works \cite{AI,AA,IPS}, are essentially   disjoint from the ones which led us to define the power series $T_{\tilde W}(s,q)$.
Cigler originally was   interested in finding  $q$-analogues of bivariate Fibonacci and Lucas polynomials having  simple formulas and  also satisfying simple recurrences; see the discussion in \cite[\S1]{Cigler1}. Cigler found that these aims could be accomplished  simultaneously by introducing two additional parameters rather than just one, and the $q$-Chebyshev polynomials of the first (respectively, second) kind arise as  natural special cases 
of his quadrivariate analogues of the Lucas (respectively, Fibonacci) polynomials.

It is clear from the recurrence defining $T_n(x,s,q)$ that $q$-Chebyshev polynomials have nonnegative integer coefficients.
Thus, from Theorem \ref{main2-thm} we obtain this additional corollary:

\begin{corollary}\label{cor2}
For all $n>0$ it holds that $T_{\tilde S_n}(s,q) \in \NN[s,q]$; that is, the power series $T_{\tilde S_n}(s,q)$ is   a polynomial in $s$ and $q$ with nonnegative integer coefficients.
\end{corollary}

\begin{remark}
The more general power series $T^J_{W,*}(s,q)$ is not always a polynomial; see Proposition \ref{notpol-prop}.
 It would also be interesting to find some interpretation of the    coefficients of $T_{\tilde S_n}(s,q)$ in terms of  group-theoretic or geometric information attached to $\tilde S_n$.
Cigler has given a combinatorial interpretation of the   coefficients of $T_n(x,s,q)$: these count (certain) tilings of an $n$-set with a fixed ``weight,'' where a \emph{tiling} of an $n$-set is a way of forming an $(n\times 1)$-rectangle from some combination of white $(1\times 1)$-squares, black $(1\times 1)$-squares, and $(2\times 1)$-dominoes; see  \cite[Theorem 2.4]{Cigler2}.
\end{remark}

As a final comment, we note that $T_{\tilde S_n}(s,q)$ has a well-defined limit as $n\to \infty$.
Our notion of convergence in this statement is the usual one for formal power series; that is,  a sequence of power series converges to a limit if the sequence of coefficients of any fixed degree eventually stabilizes. 
\begin{corollary}\label{limit-cor}
It holds that
\[
 \lim_{n\to \infty} T_{\tilde S_n}(s,q) 
=
\sum_{k=0}^\infty \(s^k  \cdot q^{k^2} \cdot (q;q)_k^{-1}\cdot  (-q^{k+1};q)_\infty \)
= \prod_{k=1}^\infty (1+q^k)(1+sq^{2k-1}).
\]
\end{corollary}

\begin{proof}
The first equality follows on noting
that in $\ZZ[[s,q]]$, we have
$\ds\lim_{n\to \infty} [n]_q = \tfrac{1}{1-q}$ and $\ds\lim_{n\to \infty} \tbinom{n-k}{k}_q = \tfrac{1}{[k]_q! (1-q)^k}$.  The second equality follows (by a combinatorial argument) from  \cite[Eq.\ (19.5.1)]{HW} or (using the theory of hypergeometric series) from \cite[Eq.\ (1.3.16)]{GR}.
\end{proof}

\begin{remark}
One may also write 
$\lim_{n\to \infty} T_{\tilde S_n}(s,q)= {_0 \phi_1}(-q;q,sq) \cdot (-q;q)_\infty $.
The product formula in the corollary has a natural interpretation as the generating  function of a certain class of integer partitions; see the discussion in \cite[Chapter 19]{HW}.
%It would be interesting to know if  this product was the limiting case of  some general factorization for the  polynomials $T_{\tilde S_n}(s,q)$.
\end{remark}

It seems natural to conjecture that the phenomena we have identified  in connection with the power series $T_{\tilde S_n}(s,q)$ should fit into some larger picture.
 To investigate this possibility, the obvious next step is to compute $T_{\tilde W}(s,q)$ for the other (classical) affine Weyl groups $\tilde W$.
Outside of type $A$, are these power series still polynomials with nonnegative coefficients?
In view of Corollary \ref{main2-cor}, it would be especially interesting to know whether   the trivariate power series $x^{|\tilde S|} \cdot T_{\tilde W}(s/x^2,q)$ 
is a $q$-analogue of some  known family of orthogonal polynomials.

\subsection*{Acknowledgements}

We thank Megan Bernstein, Dan Bump, Persi Diaconis,  Angela Hicks, and Tom Koornwinder for many helpful discussions and suggestions in the course of the development of this paper.

\section{Results}\label{results-sect}

Throughout, we let $[n] = \{ i \in \ZZ : 0 < i \leq n \}$ for $n \in \ZZ$, so that $[0] = \varnothing$.

\subsection{Universal Coxeter systems}\label{ex-sect}

Let $U_n$ be the \emph{universal Coxeter group} of rank $n$, i.e., the Coxeter group generated by a set of $n$ simple generators subject to no braid relations (so that the product of any two distinct simple generators has infinite order). 
Each permutation of the simple generating set extends to an automorphism of the group, and every automorphism of $U_n$ preserving the set of simple generators arises in this way. 

The power series defined  in the introduction may all be given explicit formulas for universal Coxeter groups. In this section we derive such formulas, which are useful as examples.

\begin{proposition}\label{Uprop}
Suppose $*$ is an involution of $U_n$ which preserves the  group's simple generating set, and which fixes exactly $f$ simple generators.
Then
\[ L_{U_n,*}(q) = \tfrac{1+q^2}{1+q} \cdot \tfrac{1-(f-1)q}{1-(n-1)q^2}
\qquand
P_{U_n}(q) = \tfrac{1+q}{1-(n-1)q}
\qquand
F_{U_n,*}(q) = \tfrac{1+q}{1-(f-1)q}.
\]
%In particular, it holds that $L_{U_n,*}(q) = P_{U_n}(q^2)/ F_{U_n,*}(q)$.
\end{proposition}

\begin{proof}
Elements of length $k$ in $U_n$ are in bijection with  $k$-letter words in an alphabet of size $n$ with no equal adjacent letters.
Hence when $k \geq 1$ there are $n(n-1)^{k-1}$ elements $w \in U_n$ with $\ell(w) = k$,
which gives the formula for $P_{U_n}(q)$.
Next, 
observe that 
each 
 element in $\I_*(U_n)$ has a (unique) reduced word of the form 
\[w  = (abc\cdots z) r (z\cdots cba)^*\qquad\text{or}\qquad w' = (abc\cdots zs)(sz\cdots cba)^*\] where $a,b,c,\dots,z,r,s \in S$ and $r=r^*$ and $s\neq s^*$. 
We have $\ellinv{*}(w) = 1$ and $\ellinv{*}(w') =0$, and  there are $f\cdot (n-1)^k$ elements $w \in \I_*$ with $\ell(w) = 2k+1$ (when $k\geq 0$) 
and $(n-f)(n-1)^{k-1}$ elements $w' \in \I_*$ with $\ell(w') = 2k$ (when $k\geq 1$).
It follows that 
$ L_{U_n,*}(q) =  \tfrac{q-1}{q+1} \cdot \tfrac{fx}{1-(n-1)q^2} +  \tfrac{1-(f-1)q^2}{1-(n-1)q^2}
% = \frac{x^2+1}{x+1}\cdot \frac{1-(f-1)x}{1-(n-1)x^2}
$
which simplifies to the given expression.
Finally, it is clear that $F_{U_n,*}(q) = P_{U_f}(q)$.
\end{proof}

We give a formula for $T^{J}_{U_{n}}$ to show that this power series is not always a polynomial in $s$ and $q$. 
\begin{proposition}\label{notpol-prop}
Let $J$ be a set of $j$ simple generators in a universal Coxeter group $U_n$. Then
\[
\ds T^{J}_{U_{n}}(s,q) = 1 + \frac{ (n-j)  q (1-q) (1+s)}{(1-(j-1)q^2) (1-(n-1)q)}.
\]
\end{proposition}

\begin{proof}
Let $W = U_n$ and write $S$ for the set of simple generators in $W$.
Assume $w \in W$ is an involution which is also a minimal length $(W_J,W_J)$-double coset representative,
and 
define $\diamond \in \Aut(W)$ and $K \subset J$ so that $(w,\diamond,K)$ is a $(W,J,\id)$-double coset datum.
Since $W$ is a universal Coxeter group, either $w=1$ or $w=s_k\cdots s_2s_1s_2 \cdots s_k$ for some generators $s_i \in S$ such that $s_1s_2\cdots s_k$ is a reduced word with $s_k \notin J$.
In the first case $\diamond = \id$ and $K = J$. In the second case, there are $(n-j)(n-1)^{k-1}$ choices for the simple generators $s_i$ and for each choice we have $\ell(w) = 2k-1$ and $\ellprime(w) = 1$ and $K = \varnothing$.
It follows that 
\[ T_{U_n}^{J}(q) = P_{U_n}(q) \cdot \tfrac{ P_{U_{j}}(q^2)}{P_{U_n}(q^2)} \cdot \( \tfrac{1}{P_{U_j}(q)} + \sum_{k=1}^\infty (n-j)(n-1)^{k-1} \cdot q^{2k-1} \cdot s \cdot \tfrac{1-q}{1+q} \).
\]
By substituting the identity $ \sum_{k=1}^\infty (n-1)^{k-1}  q^{2k-1}
% = (n-j)q \sum_{k\geq 0} ((n-1)q^2)^k 
= \tfrac{q}{1-(n-1)q^2}$
and
the formula for $P_{U_n}(q)$ in Proposition \ref{Uprop},
we deduce that 
\[ T_{U_n}^{J}(q) =
  \tfrac{1+q}{1-(n-1)q}
  \cdot
    \tfrac{1+q^2}{1-(j-1)q^2}
    \cdot
      \tfrac{1-(n-1)q^2}{1+q^2}
\cdot
  \(
\tfrac{1-(j-1)q}{1+q}
+
 \tfrac{1-q}{1+q} \cdot \tfrac{(n-j)qs}{1-nq^2} 
 \)
 \]
 which simplifies to the desired equation.
\end{proof}

\subsection{Affine symmetric groups}\label{affSn-sect}

Everywhere in this section, $n$ denotes a positive integer with $n\geq 2$.
Recall  our abstract definition of $\tilde S_{n}$ from the introduction: this is  the Coxeter group with the Coxeter diagram \eqref{diagrams} when $n\geq 3$, or the universal Coxeter group $ U_2$  described in Section \ref{ex-sect} when $n=2$. (While our main results hold with $\tilde S_1 $ defined to be the  group $\{1\}$,  not all statements in this section will make sense in this trivial case.)

As  discussed in  \cite[\S1.12]{Lu}, one may
 construct $\tilde S_n$ 
as a group of permutations in the following way.
Consider  the subgroup of bijections $w: \ZZ \to \ZZ$ satisfying the following two conditions: 
\[w(i+n) = w(i)+n\text{ for all $i \in \ZZ$}\qquand \sum_{i \in [n]} w(i) = \sum_{i \in [n]} i.\]
We call the group of such permutations the \emph{affine symmetric group} (of rank $n$).
When $n\geq 2$ there is a unique isomorphism from $\tilde S_n$ to the affine symmetric group which maps the simple generators  $s_i \in \tilde S_n$ for $i \in \{0,1,\dots,n-1\}$ to the 
permutations of $\ZZ$ given by 
\be\label{zperm-eq}  j \mapsto \begin{cases} j+1 & \text{if }j\equiv i \modu n) \\ j-1 &\text{if }j \equiv i+1 \modu n) \\ j &\text{otherwise.}\end{cases}
\ee
We  identify $\tilde S_n$ with its image under this isomorphism, and let $s_i \in \tilde S_n$ 
for $i \in \{0,1,\dots,n-1\}$ 
denote the permutation \eqref{zperm-eq}.

Given $w \in \tilde S_n$, define 
$\inv(w)  $ as the set 
\[ \inv(w)  = \left\{ (i,j) \in \ZZ\times \ZZ : i <j\text{ and }w(i)>w(j)\right\}.\]
There is now this description of the length function of $\tilde S_n$:

\begin{proposition}[Lusztig \cite{Lu}]
If  $w \in \tilde S_n$ then $\ell(w)$ is the number of equivalence classes in $\inv(w)$, under the equivalence relation  $\sim$ on $\ZZ\times \ZZ$ generated by setting $(i,j) \sim (i+n,j+n)$ for all $i,j$.
\end{proposition}

It is straightforward to derive the following statements from this proposition:

\begin{corollary}\label{ell-cor}
If $w \in \tilde S_n$, then
$
 \ell(w) = \# \{ (i,j) \in \ZZ\times [n] : i<j\text{ and }w(i)>w(j)\}.
 $
\end{corollary}

%\begin{proof}
%Each equivalence classes in $\inv(w)$ under the relation $\sim$ contains a unique representative in the set $ \{ (i,j) \in \ZZ\times [n] : i<j\text{ and }w(i)>w(j)\}$.
%\end{proof}

\begin{corollary}\label{des-cor}
If $w \in \tilde S_n$ and $0 \leq i < n$, then 
 $\ell(ws_i) < \ell(w)$ if and only if $w(i) > w(i+1)$.
\end{corollary}

Recall the definition of twisted absolute length function $\ellinv{*}$ from the introduction, and note that we write $\ellprime = \ellinv{\id}$.
As the first new result in this section, we  establish a formula for $\ellprime $ on %the set of involutions in 
$\tilde S_n$.

\begin{proposition}\label{ell*-prop} If $w \in \tilde S_n$ is an involution, then
$ \ellprime (w) = \frac{1}{2}(n-\#\{ i \in [n] : w(i)=i\})$.
\end{proposition}

\begin{proof}
Let $f(w) =  \frac{1}{2}(n-\#\{ i \in [n] : w(i)=i\})$ for involutions $w \in \tilde S_n$. To show that $f = \ellprime $, it suffices to check that the function $f$ has the three defining properties of $\ellprime $ given in the introduction. The first  property, asserting that $f(1) = 0$, clearly holds.

To show the second property,   that $f$ is constant on each conjugacy class of involutions, note that 
the size of $\{ i \in [n] : w(i) = i\}$ is equal to the number of equivalence classes in the set $\{ i \in \ZZ : w(i) = i\}$ under the equivalence relation $\sim$ on $\ZZ$ generated by setting $i \sim i+n$ for all $i$.
The number of such equivalence classes is unchanged if we replace $w$ by $xwx^{-1}$ for some $x \in \tilde S_n$, since $w(i+kn) = i+kn$ for all $k$ if and only if $xwx^{-1}(x(i)+kn) = x(i)+kn$ for all $k$.
Thus $f(xwx^{-1}) = f(w)$ for all $x \in \tilde S_n$, as required.

For the third property, suppose $w \in \tilde S_n$ is an involution and $i \in \{0,1,\dots,n-1\}$ is such that $s_iw =ws_i$. Then $ws_i$ is also an involution, and we must show that $f(ws_i) - f(w) = \ell(ws_i) - \ell(w)$.
Without loss of generality we may assume that $\ell(ws_i) < \ell(w)$, so that
$w(i) > w(i+1)$ by Corollary \ref{des-cor}.
Let $j = w(i)$ and $k=w(i+1)$. Then 
\[j = ws_i(i+1) = s_iw(i+1) = s_i(k)>k,\]   so
we must have   $j=k+1$ and $k = i+mn$ for some $m \in \ZZ$.
We deduce that in fact $m=0$, since
\[ i+1 = w^2(i+1)  =w(k) = w(i+mn) = w(i) + mn = j+mn = i+1+2mn,\]
so $w(i) = i+1$ and $w(i+1) = i$. It follows from this that if $t \in \ZZ$ then $ws_i(t) = t$ if and only if 
\[ w(t) = t \qquad\text{or}\qquad t \equiv i \modu n) \qquad\text{or}\qquad t \equiv i+1\modu n).\]
Consequently, the number of equivalence classes in $\{ t \in \ZZ : ws_i(t) = t\}$ is exactly two greater than 
the number of equivalence classes in $\{t \in \ZZ : w(t) = t\}$,
so $f(ws_i) - f(w) = - 1 = \ell(ws_i)-\ell(w)$ and we conclude that $f = \ellprime $.
\end{proof}

Let $\tau_n : \ZZ \to \ZZ$ be the map given by  
$
\tau_n(i) = n+1-i$ for $i \in \ZZ$.
%Observe that $\tau_n = \tau_n^{-1}$.
Although $\tau_n$ is not itself an element of the affine symmetric group,
conjugation by $\tau_n$ defines an automorphism of $\tilde S_n$, which we denote in this and the next section by $* \in \Aut(\tilde S_n)$; that is, we let
\be
\label{*def}
 w^* = \tau_n \cdot w \cdot \tau_n\qquad\text{for }w \in \tilde S_n.
 \ee
Observe that $s_0^* = s_0$ while $s_i^* = s_{n+1-i}$ for each $i \in [n-1]$. Thus, $*$ acts to flip the Coxeter diagram of $\tilde S_n$ given in \eqref{diagrams} about its vertical axis of symmetry, 
and in particular
%and explicitly we may describe $w^*$ as the map 
%\[ w^* : i \mapsto n+1 - w(n+1-i).\]
 $*$ 
 %is an involution of $\tilde S_n$ which 
 preserves the set  of simple generators in $\tilde S_n$.
A permutation  $w \in \tilde S_n$ belongs to the set of twisted involutions $\I_*(\tilde S_n)$ if and only if $(w \tau_n)^2 = 1$ or, equivalently, $ (\tau_n w)^2 = 1$.

While most of our results will only concern the ordinary involutions in $\tilde S_n$, twisted involutions relative to the automorphism $*$ will arise naturally in the next section. 
For completeness, we derive here a formula for  the twisted absolute length function attached to this involution.

\begin{proposition}\label{ell**-prop} If $w \in \I_*(\tilde S_n)$ then
$\ds  \ellinv{*}(w) = \left\lfloor \frac{\#\{ i \in [n] :  w(i) \equiv 1-i \modu n)\}}{2} \right\rfloor.$
\end{proposition}

\begin{proof}
Let $f(w) =  \left\lfloor \#\{ i \in [n] :  w(i) \equiv 1-i \modu n)\} / 2 \right\rfloor$ for twisted involutions  $w \in \I_*(\tilde S_n)$. We argue that  $f=\ellinv{*}$ as in the proof of Proposition \ref{ell*-prop}, by showing that $f$ has the three properties given in the introduction which uniquely determine $\ellinv{*}$.

It holds that $f(1) = 0$ since the set $\{ i \in [n] :  i \equiv 1-i \modu n)\}$ has at most one element.
Let $g(w)$ for $w \in \I_*(\tilde S_n)$ denote the number of equivalence classes in the set $\{ i \in \ZZ : (\tau_n w)(i) \equiv i \modu n)\}$ under the equivalence relation on $\ZZ$ generated by setting $i \sim i+n$ for all $i$,
and observe that $f(w) = \lfloor g(w)/2 \rfloor$.
Since $\tau_n( x^*wx^{-1})= x(\tau_n w) x^{-1}$ for all $x \in \tilde S_n$ and $w \in \I_*(\tilde S_n)$,
it follows as in the proof of Proposition \ref{ell*-prop} that the function $g$ is constant on $*$-twisted conjugacy classes, so the same is true of $f$. 

Finally, suppose $w \in \I_*(\tilde S_n)$  and $i \in \{0,1,\dots,n-1\}$ are such that $s_i^*w =ws_i \in \I_*(\tilde S_n)$. To show that $f=\ellinv{*}$, it suffices to check that $f(ws_i) - f(w) = \ell(ws_i) - \ell(w)$.
Without loss of generality assume $\ell(ws_i) > \ell(w)$, so that
$w(i) < w(i+1)$ by Corollary \ref{des-cor}. In this case we have $(\tau_nw)(i) > (\tau_n w)(i+1)$,
so since $(\tau_n w)^2=1$,
it 
follows exactly as in the proof of Proposition \ref{ell*-prop} that $(\tau_nw)(i) \equiv i+1 \modu n)$ and $(\tau_nw)(i+1) \equiv i \modu n)$. One checks that consequently $g(ws_i) = g(w)+2$, which implies that $f(ws_i) -f(w) = 1 =\ell(ws_i) - \ell(w)$ as required.
\end{proof}

\begin{remark} Results of MacDonald  \cite{MacDonald} (see also \cite[Theorem 3.10]{Steinberg}) show how to compute the power series $F_{W,*}(q)$ when $W$ is any finite or affine Weyl group. For example, if  $* \in \Aut(\tilde S_n)$ is given by \eqref{*def}, 
then one can check using \cite{MacDonald} that
\[
%F_{S_n,*}(q)  
%=
%W_{B_l}(q^2,q^e)
% =  
% \prod_{k=1}^{n} \tfrac{1-(-q)^k}{1-(-1)^kq} 
% \quand
 F_{ \tilde S_n,*}(q) 
 =
 \tilde W_{C_l}(q^2,q^e,q)
 =
\prod_{k=1}^{n-1} \tfrac{1-(-q)^{k+1}}{(1+(-1)^k q)(1+(-q)^k)}
\]
where in the second expression  
we use the notation of  \cite[\S3]{MacDonald}
and
define 
$ l = \lfloor \tfrac{n}{2}\rfloor $
and
$ e=2 - (-1)^n
$.
Theorem \ref{L6-thm} then implies
\[
L_{\tilde S_{n},*}(q) 
%= \(\prod_{k=1}^{n-1} \frac{1}{1-(-q)^k}\)L_{S_n,*}(q)
=
\prod_{k=1}^{n-1} \tfrac{1+(-q)^{k+1}}{(1-(-1)^kq)(1-(-q)^{k})}
.\] 
It would interesting to know the bivariate analogue of this formula, that is, for power series $T_{\tilde S_n,*}^J(s,q)$ with $J = \{ s_i : i\in [n-1]\}$ and $*$ as in \eqref{*def}.
Many of the details required to compute this are provided by  results in the next section, but we will not   undertake this calculation.
 \end{remark}

\subsection{Combinatorics of  coset representatives}

Continue to let $n\geq 2$ be an integer.
%Recall that the Coxeter group $S_n$ given abstractly by \eqref{diagrams} is typically defined  as the group of permutations of the set $[n]$. 
For each permutation $w\in S_n$, there exists a unique element $\tilde w \in \tilde S_n$ whose restriction to $[n]$ coincides with $w$.
The map $w \mapsto \tilde w$ is a group isomorphism from $S_n$ to  
 the standard parabolic subgroup of $\tilde S_n$ generated by $\{s_1,s_2,\dots,s_{n-1}\}$,
 and we identify $S_n$ with its image under this map. This identification makes consistent our convention of writing $s_i$ for the simple generators of both $S_n$ and $\tilde S_n$,
and allows us to
speak 
of
 $(S_n,S_n)$-double cosets in $\tilde S_n$. 
 
Each $(S_n,S_n)$-double coset contains a unique element of minimal length (see   \cite[\S2.1]{GP}), which by Corollary \ref{des-cor} may be characterized by the following condition.

\begin{proposition}\label{minlen-prop} A permutation $w \in \tilde S_n$ is the unique element of minimal length in its $(S_n,S_n)$-double coset if and only if 
$w(1) < w(2) < \dots < w(n)$ and  
$w^{-1}(1) < w^{-1}(2) < \dots < w^{-1}(n).$
\end{proposition}

Throughout this section $*$ remains the involution of $\tilde S_n$ defined by \eqref{*def}.
For involutions or $*$-twisted involutions, the preceding proposition simplifies to the following:

\begin{corollary}\label{minlen-cor}
Suppose $w \in \tilde S_n$ is such that $w^{-1} \in \{ w,w^*\}$. Then $w$ is the unique element of minimal length in its $(S_n,S_n)$-double coset if and only if $w(1) < w(2) <\dots <w(n)$.
\end{corollary}

%\begin{proof}
%The two inequalities in Proposition \ref{minlen-prop} are equivalent when $w^{-1} \in \{ w,w^*\}$. 
%\end{proof}

Let $\Imin_n$ (respectively, $\Imin'_n$) denote the set of involutions (respectively, $*$-twisted involutions) that are minimal length $(S_n,S_n$)-double coset representatives in $\tilde S_n$; i.e., set
\be\label{omega-def}
\ba
\Imin_n &= \left\{ w \in \tilde S_n : w^2=1 \text{ and } \ell(s_i w) = \ell(ws_i) = \ell(w)+1 \text{ for all } i \in [n-1]\right\}
\\
\Imin'_n &= \left\{ w \in \tilde S_n : (\tau_n w)^2=1\text{ and } \ell(s_i w) = \ell(ws_i) = \ell(w)+1 \text{ for all } i \in [n-1]\right\}
.
\ea
\ee
Note that requiring these sets to consist of minimal length double coset representatives is slightly redundant, as  a ($*$-twisted) involution in $\tilde S_n$ is a minimal length $(S_n,S_n)$-double coset representative if and only if it is  a minimal length left (equivalently, right) $S_n$-coset representative. Recall that $\Imin_n$ indexes the sum defining $T_{\tilde S_n}(s,q)$.

The main goal of this section is to derive from Corollary \ref{minlen-cor} a more concrete description of $\Imin_n$ and $\Imin'_n$.
In more detail,
given a permutation $w \in \tilde S_n$, define 
\[\lambda_i(w) = \lfloor \tfrac{w(i)-1}{n} \rfloor
\qquand \lambda(w) = (\lambda_1(w),\lambda_2(w),\dots,\lambda_n(w)).\]
The set of integers $\ZZ$ decomposes as a disjoint union of shifted copies of $[n]$, and the number $\lambda_i(w)$ records which copy contains the image of $i$ under $w$. In particular, $\lambda(w)$ may be alternatively defined as the sequence of numbers such that
 $w(i) \in n\cdot \lambda_i(w) + [n]$ for each $i \in [n]$. %; in particular, $\overline w \in S_n$ is  the map $i \mapsto n\cdot \lambda_i(w) + w(i)$.
We will first show that the sequence $\lambda(w)$ uniquely determines the element $w \in \Imin_n\cup \Imin'_n$, and then   derive a formula for the length of $w$ in terms of $\lambda(w)$.

Given any integer sequence $a = (a_1,a_2,\dots,a_n)$, let $\beta(a)$ be the sequence of indices 
which record the final positions of the contiguous blocks of equal entries in $a$; explicitly, 
$\beta(a) = (b_1,b_2,\dots,b_m) $ is  the strictly increasing sequence of positive integers  such that, setting $b_0=0$, we have
\[  a_{b_{i-1}+1} = a_{b_{i-1}+2} =\dots = a_{b_{i}} \neq a_{b_{i}+1}\text{ for all }i \in [m-1]
\qquand b_m = n.\]
 For example, if $a = (0,0,-7,-7,-7,2)$ then $\beta(a)= (2,5,6).$
 If $a$ is the empty sequence then  we define the sequence $\beta(a)$ to be likewise  empty.
Recall that a sequence $a$ is \emph{weakly increasing} if $a_1 \leq a_2\leq \dots \leq a_n$ and \emph{antisymmetric} if $a_i + a_{n+1-i} = 0$ for all $i \in [n]$.
Observe that the central term in an antisymmetric sequence  of odd length must be zero.

 \begin{lemma}\label{a-lem}
 Fix an integer sequence $a=(a_1,a_2,\dots,a_n)$ and 
let $\beta(a) = (b_1,b_2,\dots,b_m)$. Set $b_0 = 0$ and
  for each $i \in [m]$ define $v_i = a_{b_{i-1}+1} =a_{b_{i-1}+2}=\dots = a_{b_i}$ and also define
 \[  I_i = \{ k \in [n] : b_{i-1} < k \leq b_i\}\qquand J_i  = n+1-I_i.\]
  Finally let $w_a : \ZZ \to \ZZ$ 
be the unique  map   $\ZZ \to \ZZ$ which 
 restricts to an order-preserving bijection $nk+I_i \to n(k + v_i) + J_i$ for each $i \in [m]$ and $k \in \ZZ$.
 The following properties then hold:
 \ben
% \item[(a)] There exists a unique permutation  $w_a$ of $\ZZ$ which (1) satisfies $w_a(t+n) = w_a(t) + n$ for all $t \in \ZZ$ and   (2) restricts to an order-preserving bijection $I_i \to na_i+J_i$ for each $i \in [m+1]$.
 
 \item[(a)] $w_a  \in \tilde S_n$ if and only if  $\sum_{i=1}^n a_i = 0$, and in this case $\lambda(w_a)=a$.
 
 \item[(b)] $w_a \in \Imin'_n$ if and only if $a$ is weakly increasing and $\sum_{i=1}^n a_i = 0$.
 
 \item[(c)] $w_a \in \Imin_n$ if and only if $a$ is weakly increasing and antisymmetric.

 \een
 \end{lemma}
 
 \begin{remark}
 Before beginning the proof of this lemma, we give an example of the map $w_a$.
 Let $n=3$ and $a=(3,3,-6)$, so that $\beta(a) = (2,3)$ and $(v_1,v_2) = (3,-6)$.
 Then
 \[ I_1 = \{1,2\} \leftrightarrow J_1 = \{2,3\} \qquand  I_2 = \{3\} \leftrightarrow J_2 = \{1\}\]
 and $w_a$ is the unique  permutation of $\ZZ$ with $w_a(i+3) = w_a(i)+3$ for all $i$ and with
 \[w_a(1) =11 = 2 + nv_1 \qquand w_a(2) = 12 = 3 + nv_1 \qquand w_a(3) = -17 = 1 + nv_2.
 \]
 \end{remark}
 
 \begin{proof}
 Observe that both
$\{ nk + I_i : (i,k) \in [m]\times \ZZ\}$ and $\{ nk+J_i : (i,k)  \in [m]\times \ZZ\}$ are partitions of $\ZZ$ into disjoint subsets, and that each set in the first partition is in bijection via $w_a$ with exactly one set in the second. It follows that the map $w_a$ is in fact a permutation of $\ZZ$.

By construction $w_a(t+n) = w_a(t) + n$ for all $t \in \ZZ$,
so, in light of the definition of $\tilde S_n$ in Section \ref{affSn-sect},
part (a) follows by observing that 
\[ \sum_{i \in [n]} w_a(i) = \sum_{i \in [m]} \sum_{t \in I_i} w_a(t) = \sum_{i \in [m]}  \sum_{t \in J_i}(nv_i+t)
%= \sum_{ i\in [n]}  i + n\sum_{i \in [m]} (b_i-b_{i-1})v_i 
= \sum_{i \in[n]} (i + na_i).\] 
For the remaining parts, assume $\sum_{i =1}^n a_i = 0$ so that $w \in \tilde S_n$.
It is clear from the definition of $w_a$ that  $\lambda(w_a) = a$
and that 
 $w_a(1) < w_a(2) < \dots < w_a(n)$ if and only if 
$a$ is weakly increasing.
Therefore, by Corollary \ref{minlen-cor}, to prove (b)  it suffices to check that $w_a^* = w_a^{-1}$, and to prove (c) it suffices to check that $w_a=w_a^{-1}$ if and only if $a$ is antisymmetric.

For part (b), note that the map $\tau_n$ given before \eqref{*def} restricts to  an order-reversing bijection $I_i \to J_i$ and $J_i \to I_i$ for each $i \in [m]$. Using this observation, one checks that  $w_a^* = \tau_n \cdot w_a \cdot \tau_n$  restricts to an order-preserving bijection $n(k+v_i) + J_i \to nk + I_i$ for each $i \in [m]$ and $k \in \ZZ$.  As $w_a^{-1}$ is clearly the unique map $\ZZ \to \ZZ$ with this description, we must have $w_a^* = w_a^{-1}$ as desired.
For part (c), note that $w_a=w_a^{-1}$ if and only if 
  $v_i = -v_{m+1-i}$ and $J_i = I_{m+1-i}$ for each $i \in [m]$, which    holds precisely when $a$ is antisymmetric.
 \end{proof}

 We may now prove the main result of this section.
  
  \begin{theorem}\label{Omega-thm}
The maps
$w \mapsto \lambda(w)
$
and 
$a\mapsto w_a$ %(with $w_a$ as in Theorem \ref{a-thm}) 
are inverse  bijections between the following:
\begin{itemize}
\item[(i)] $\Imin'_n \leftrightarrow \{\text{ weakly increasing sequences of $n$ integers whose terms sum to zero }\}$.
\item[(ii)] $  \Imin_n \leftrightarrow \{\text{ weakly increasing, antisymmetric  sequences of $n$ integers }\}$
\end{itemize}
\end{theorem}

\begin{proof}
Suppose $w \in \tilde S_n$ is the unique element of minimal length in its $(S_n,S_n)$-double coset,
and let $a=(a_1,a_2,\dots,a_n) = \lambda(w)$. Observe that since $w(1) < w(2) < \dots <w(n)$ by Proposition \ref{minlen-prop}, the sequence $a$ must be weakly increasing.

To prove the theorem it suffices by Lemma \ref{a-lem} to show that $w=w_a$ when $w$ belongs to $\Imin'_n$ or $\Imin_n$.
Towards this end, let $\beta(a) = (b_1,b_2,\dots,b_m)$ and for each  $i \in [m]$ define  
\[ v_i \in \ZZ \qquand I_i, J_i \subset [n]\]  exactly as in Lemma \ref{a-lem},
and additionally let $K_i = -nv_i + w(I_i)$.
Observe that 
$\tau_n(I_i) = J_i$ and 
     $[n] = K_1 \cup K_2 \cup \dots \cup K_m$, where the union is disjoint.
Since $w(i+n) = w(i)+n$ for all $i \in \ZZ$ and  since $w(1) < w(2) < \dots <w(n)$,  the map $w$  restricts to an order-preserving bijection $I_i \to nv_i + K_i$ for each $i \in [m]$.
Therefore, 
to prove that $w=w_a$ when $w \in \Imin'_n \cup \Imin_n$, we just need to show in this case that $ J_i = K_i$ for each $i \in [m]$.

For this, first suppose that $w \in \Imin'_n$ so that $w^{-1} = w^* = \tau_n \cdot w \cdot \tau_n$. Then $(w\tau_n)^2 = 1$
so
 the composition 
\[ J_i \xrightarrow{\tau_n} I_i \xrightarrow{w} nv_i + K_i \xrightarrow{\tau_n} -nv_i + \tau_n(K_i) \xrightarrow{w} -nv_i + (w\tau_n)(K_i)
\]
is the identity map and in particular $J_i = -nv_i + (w\tau_n)(K_i)$. This implies that \[ (w\tau_n)(K_i) = nv_i + J_i \subset nv_i + [n],\] so we must have $\tau_n(K_i) \subset I_i$ which implies in turn  that $K_i \subset \tau_n(I_i) = J_i$.
Thus $K_i \subset J_i$ for all $i \in [m]$, so since both $\{ J_i\}$ and $\{K_i\}$ are partitions of $[n]$ into disjoint subsets, it must hold that $J_i = K_i$ for all $i \in [m]$ as required.

Next suppose that $w \in \Imin_n$ so that $w^2=1$. Recall that  $\overline{w}$ denotes the image of $w$ under the homomorphism $\tilde S_n \to S_n$, and observe that $\overline w$ here has the formula  $ i \mapsto  w(i) - na_i$ for $i \in [n]$. Since $\overline w$ is also an involution,
 for each $i \in [n]$ it holds that
\[ i = w^2(i) = w(\overline{w}(i)+na_i) =  \overline{w}^2(i) + n(a_i + a_{\overline w(i)})= i +  n(a_i + a_{\overline w(i)}).\]
Hence $a_i = -a_{\overline{w}(i)}$, so whenever $v \in \ZZ$ appears in the sequence $a$, the number $-v$ also appears, with the same multiplicity.
As $a=(a_1,a_2,\dots,a_n)$ is already weakly increasing, this observation implies that $a$ is   antisymmetric.
Thus $a_i = -a_{n+1-i}$ for each $i \in [n]$, so (by the definition of $v_i$, $I_i$, and $J_i$ in Lemma \ref{a-lem}) it must hold that $v_i = -v_{m+1-i}$ and $J_i = I_{m+1-i}$ for each $i \in [m]$.
To now deduce that $w=w_a$, we note that 
since $w^2=1$,  we have must $I_i = nv_i + w(K_i)$ for each $i \in [m]$ so \[w(K_i) = -nv_i + I_i = nv_{m+1-i} + I_i \subset nv_{m+1-i} + [n]\] which implies $K_i \subset I_{n+1-i} = J_i$.
This containment  can hold for all $i \in [m]$ only if $J_i = K_i$, again since   $\{ J_i\}$ and $\{ K_i\}$ are both partitions of $[n]$.
\end{proof}

While not every involution in $\tilde S_n$ is a $*$-twisted involution, the preceding theorem  shows that the following inclusion does hold, and is strict when $n>2$:
\begin{corollary}
For all $n\geq 2$ it holds that $\Imin_n \subset \Imin'_n$.
\end{corollary}

Finally, we may give the promised length formula for $w \in \Imin'_n$ in terms of $\lambda(w)$.

  \begin{lemma}\label{ellOmega-cor} 
  %If $a=(a_1,a_2,\dots,a_n)$ is a weakly increasing integer sequence with $\sum_{i=1}^n a_i = 0$ and $w_a \in \tilde S_n$ is defined as in Lemma \ref{a-lem}, 
  If $w \in \Imin'_n$ and $\lambda(w)=(a_1,a_2,\dots,a_n)$, then 
  \[\ell(w) = \sum_{1\leq i<j \leq n} \max(a_j-a_i-1,0)
 =  \sum_{1\leq i <j \leq n} \Bigl(a_j-a_i+  \One(a_i = a_j)\Bigr) - \tbinom{n}{2}\]
where $\One(x = y)$ denotes the function which is 1 if $x=y$ and 0 otherwise.
 \end{lemma}
 
 \begin{proof}
 We only show that $\ell(w) = \sum_{i<j} \max(a_j-a_i-1,0)$ since the second formula follows easily from this.
 Let $a=\lambda(w) = (a_1,a_2,\dots,a_n)$. By Theorem \ref{Omega-thm}, the sequence $a$ is then weakly increasing and $\sum_{i=1}^n a_i = 0$, and $w=w_a$.
 For each $i,j \in [n]$ let $E_{i,j} = \{ (j-nk,i) : k \in \NN \} \cap \inv(w_a)$. 
From Corollary \ref{ell-cor} we have $\ell(w)=\ell(w_a) = \sum_{(i,j) \in [n]\times[n]} |E_{i,j}|$,
 so 
 %since $a$ is weakly increasing
  it suffices to show that 
 \be\label{E-eq} E_{i,j} = \{ (j-nk,i) : 0 < k < a_j-a_i\}.\ee
 This certainly holds if 
 $j\leq i$, since then  $ w_a(j) \leq w_a(i)$ so $w_a(j-nk)=w(j)-nk \leq w_a(i)$ for all $k \in \NN$, whence $E_{i,j} = \varnothing$.
Alternatively, assume $i<j$ so that $a_i \leq a_j$.
 The definition of $w_a$ in Lemma \ref{a-lem} then implies the following:
\begin{itemize}
\item 
 If $a_i = a_j$ then $w_a(j) - w_a(i) < n$ so $w_a(j-nk) = w_a(j)-nk <  w_a(i)$ for all $k\geq 1$.
 
 \item
 If   $a_i<a_j$ then
 $w_a(j-nk) > w_a(i)$ if and only if $k<a_j-a_i$
 since
  \[   w_a(i) + n(a_j-a_i-1) < w_a(j) < w_a(i) + n(a_j-a_i).\]
 \end{itemize}
 Since  $j-nk < i$ if and only if $k\geq 1$ (as we assume $1\leq i<j \leq n$),
 these observations together imply that \eqref{E-eq} holds, which is what we needed to show.
% $ E_{i,j} = \{ (j-nk,i) : 0<k < a_j-a_i\}$. Hence $|E_{i,j}| =\max(a_j-a_i-1,0)$
 \end{proof}

 \subsection{Some calculations}
 
Our results in the previous section all apply to the set of twisted involutions $\Imin'_n$  given by \eqref{omega-def}. Here we derive some more specific formulas which seem not to have simple analogues outside the proper subset of involutions $\Imin_n \subset \Imin'_n$. 
Throughout,  let $n \in \ZZ$  with $n\geq 2$.
  
 Fix a weakly increasing, antisymmetric integer sequence $a=(a_1,a_2,\dots,a_n)$ and  let $\beta(a) = (b_1,b_2,\dots,b_m)$. We attach three additional integer sequences to this data.
 First,  
we define $\beta^-(a)$ as  the initial subsequence of $\beta(a)$ given by truncating $\beta(a)$ at the point where its terms become nonnegative. In other words, if $m=1$ (which occurs if and only if $a= (0,0,\dots,0)$) then $\beta^-(a)$ is the empty sequence,
 while  if $m>1$ then   
\[ \beta^-(a) = (b_1,b_2,\dots,b_{m_0})\qquad\text{where $m_0 \in [m-1]$ is such that $b_{m_0} < 0 \leq b_{m_0+1}$}.
\] 
Note that if $a$ has at least one negative term, then the index $m_0$ in this definition exists and is unique since $a$ is weakly increasing and antisymmetric. 
Both the length of $\beta^-(a)$ and its final term are less than or equal to $ \lfloor \frac{n}{2}\rfloor$  for the same reason.
Next, set \[\mu^-(a) = (b_1,b_2-b_1,b_3-b_2,\dots,b_{m_0}-b_{m_0-1}).\]
Observe that the entries of this sequence are the successive multiplicities of the negative integers appearing in  $a$.
  Finally, define
\[ \delta^-(a) = (a_{b_2}-a_{b_1}, a_{b_3} - a_{b_2},\dots, a_{b_{m_0}}-a_{b_{m_0-1}},-a_{b_{m_0}}).\]
If $\Delta$  is the difference operator on sequences
given by $\Delta: (\sigma_i)_{i \in [k]} \mapsto (\sigma_{i+1} - \sigma_i)_{i \in [k-1]}$,
then $\delta^-(a)$ is what we get by applying $\Delta$ to $(a_1,a_2,\dots,a_{m_0},0)$ and then omitting all terms which are zero.
Before proceeding, let us illustrate these definitions with an example.

\begin{example}
If $a = (-6,-5,-5,-5,-2,0,2,5,5,5,6)$ then 
$ \beta(a) = (1,4,5,6,7,10,11)$
and
\[
 \beta^-(a) = (1,4,5)\qquand
 \mu^-(a) =(1,3,1)\qquand \delta^-(w) = (1,3,2).\]
If $a = (-6,-5,-5,-5,-2,2,5,5,5,6)$ then 
$ \beta(a) = (1,4,5,6,9,10)$
while $\beta^-(a)$, $\mu^-(a)$, and $\delta^-(a)$ are the same as before.
\end{example}

In view of Theorem \ref{Omega-thm}, we may transfer these statistics to  
involutions $w \in \Imin_n$ by setting
\[ \beta^-(w) = \beta^-(\lambda(w))
\qquand
\mu^-(w) = \mu^-(\lambda(w))
\qquand
\delta^-(w) = \delta^-(\lambda(w)).\]
As $w$ varies over all elements of $\Imin_w$, the sequence $\mu^-(w)$  can be any finite sequence of positive integers whose sum is at most $\lfloor \frac{n}{2}\rfloor$, while
 $\delta^-(w)$ can be any $m$-tuple of positive integers. In particular, we note the following:

\begin{lemma} \label{delta-lem}
Fix  $m \in \NN$ and  a sequence of positive integers  $c=(c_1,\dots,c_m)$ with $\sum_{i=1}^m c_i \leq \lfloor \frac{n}{2}\rfloor$.
The map 
$w \mapsto \delta^-(w)$
is then a bijection 
$ \{ w \in \Imin_n : \mu^-(w) = c \} \to \{ \text{ $m$-tuples of positive integers }\}.$
\end{lemma}

\begin{proof}
It is clear that $\mu^-(w)$ and $\delta^-(w)$ uniquely determine $\lambda(w)$, and hence $w$ by Theorem \ref{Omega-thm}, so the given map is injective. To show surjectivity, it is enough by Theorem \ref{Omega-thm} to construct a weakly increasing, antisymmetric sequence $a \in \ZZ^n$ with $\mu^-(a) = c$ and $\mu^-(a) =d$ when given   an arbitrary $m$-tuple of positive integers 
 $d=(d_1,\dots,d_m)$.
For this let
$e_{k} =  d_k+d_{k+1}+\dots+d_m$ for  $k \in [m]$ and  set  $z=n-2\sum_{i=1}^m c_i$;
one then checks that the following $n$-tuple suffices:
\[ (\underbrace{-e_1,\dots, -e_1}_{c_1\text{ entries}},\underbrace{ -e_2,\dots,-e_2}_{c_2\text{ entries}}, \dots, \underbrace{-e_m,\dots,-e_m}_{c_m\text{ entries}}, \underbrace{0,\dots,0}_{z\text{ entries}}
,
\underbrace{e_m,\dots,e_m}_{c_m\text{ entries}},
\dots,
\underbrace{e_2,\dots,e_2}_{c_2\text{ entries}},
\underbrace{e_1,\dots,e_1}_{c_1\text{ entries}}).
\]
\end{proof}

As our first application of this new notation, we prove this sequel to Lemma \ref{ellOmega-cor}.

\begin{proposition}\label{ellOmega-prop} Let $w \in \Imin_n$, write $\beta^-(w) = (b_1,b_2,\dots,b_m)$ and $\mu^-(w) = (c_1,c_2,\dots,c_m)$ and $\delta^-(w) = (d_1,d_2,\dots,d_m)$, and  set $z = n-2\sum_{i=1}^m c_i$.
It then holds that 
\[
 \ellprime (w) = \sum_{i=1}^m c_i
 \qquand
 \ell(w) = \tbinom{z}{2}- \tbinom{n}{2}+ 2\sum_{i=1}^m \Bigl( b_i(n-b_i)   d_i  + \tbinom{c_i}{2} \Bigr) .\]
\end{proposition}

\begin{remark} Observe that $ \sum_{i=1}^m c_i = b_m$ if $m>0$. When $m=0$ (which occurs only when $w=1$) we interpret the sum to be zero and set $z=n$;
our formulas then reduce to $\ellprime (w) = \ell(w) = 0$.
 %With this convention,  the proposition holds even when $w=1$.
\end{remark}

\begin{proof}
As just noted in the remark, both formulas hold trivially when $w=1$. Assume $w \in \Imin_n\setminus \{1\}$ so that $m\geq 1$.
Since $w=w_a$ for $a = \lambda(w)$ by Theorem \ref{Omega-thm}, it follows from the definition of $w_a$ in Lemma \ref{a-lem} that if $i \in [n]$ then $w(i) = i$ 
if and only 
if
$b_m < i \leq n-b_m$. (Note that the index $m$ here is not the same as the one in Lemma \ref{a-lem}.)
 Hence $\#\{ i \in [n] : w(i) = i\} = n-2b_m$ so by Proposition \ref{ell*-prop}
we have $\ellprime (w) = b_m=\sum_{i=1}^m c_i $.

To derive the given formula for $\ell(w)$, 
let $\lambda(w) = (a_1,a_2,\dots,a_n)$
and
recall that
   $\mu^-(w) = (c_1,c_2,\dots,c_m)$ is by definition the list of the nonzero multiplicities of the negative numbers in   $\lambda(w)$. Since the latter $n$-tuple is antisymmetric, it follows that  $\mu^-(w)$ is also the list of  multiplicities of the positive numbers in $\lambda(w)$,
   and that
$z$ is the multiplicity of 0 in this $n$-tuple. Hence
\be\label{preform1}
\sum_{1\leq i < j \leq n} \One(a_i=a_j)
%=\tbinom{c_1}{2}  + \dots + \tbinom{c_m}{2} + \tbinom{z}{2} + \tbinom{c_m}{2} + \dots  + \tbinom{c_1}{2}
= \tbinom{z}{2} + 2\sum_{i=1}^m \tbinom{c_i}{2}.
\ee
Next, observe that
\[
\sum_{1\leq i <j \leq n} (a_j-a_i)
=
\sum_{1\leq i <j \leq n} \sum_{k=i}^{j-1} (a_{k+1}-a_k)
=
% \sum_{k=1}^{n-1}  \sum_{1\leq i \leq k < j\leq n} (a_{k+1}-a_k) = 
\sum_{k=1}^{n-1} k(n-k) (a_{k+1}-a_k).
\]
Since $ \lambda(w)$ is antisymmetric, it follows from our definition of $\beta^-(w)$ and $\delta^-(w)$ that 
\[ a_{k+1} - a_k =  \begin{cases} d_i&\text{if $k \in \{ b_i,n-b_i\}$ for some $i\in[m]$ with $b_i< \tfrac{n}{2}$}
\\
2d_m&\text{if $k = b_m = \tfrac{n}{2}$}
\\
0&\text{if $k \notin \{b_i,n-b_i\}$ for all $i \in [m]$}.\end{cases}
\]
In particular, to derive the middle case, observe that if $k=b_m = \frac{n}{2}$, then $n$ is even and $a_{k+1} - a_{k} = -2a_{n/2} = 2d_m$. (By contrast, if $k=b_m < \frac{n}{2}$, then $a_{k+1} = 0$ so $a_{k+1}-a_{k} = -a_{b_m} = d_m$.)
Combining the two preceding equations shows that 
\be\label{preform2}\sum_{1 \leq i <j \leq n} (a_j-a_i) = 2\sum_{i=1}^m b_i(n-b_i)d_i.\ee
We now just substitute \eqref{preform1} and \eqref{preform2} into the right side of the formula in Lemma \ref{ellOmega-cor}.
\end{proof}

Consulting  \eqref{TJ-def}, we see that the only term in the definition of $T_{\tilde S_n}(s,q)$ which we cannot yet evaluate is the fixed point series $F_{K,\diamond}(q)$; we deal with this in the following proposition.

\begin{proposition} \label{FOmega-prop}
Fix $w \in \Imin_n$ 
and %write $\mu^-(w) = (b_1,b_2,\dots,b_m)$ where $m=2k-1$. 
define $\diamond \in \Aut(\tilde S_n)$ and $K \subset \{s_1,s_2,\dots,s_{n-1}\}$ by 
\[ \diamond : x \mapsto  x^\diamond = w x  w
\qquand K = \{ s_1,s_2,\dots, s_{n-1}\} \cap \{ s_1^\diamond, s_2^\diamond, \dots, s_{n-1}^\diamond\}.\]
If $ \mu^-(w) = (c_1,c_2,\dots,c_m)$ and $ z  =n-2(c_1+c_2+\dots+c_m)$,
then %$F_{K,\diamond}(q) = [z]_q! \cdot \prod_{i=1}^m [c_i]_q!$.
\[
 F_{K,\diamond}(q) =
\Bigl([c_1]_{q^2}! \cdot [c_2]_{q^2}! \cdots [c_{m}]_{q^2}!\Bigr) \cdot [z]_q!
\]
 %where $\diamond \in \Aut(\tilde S_n)$ is the automorphism $\diamond : x \mapsto  w x  w$ and $K = \{ s_1,\dots, s_{n-1}\} \cap \{ s_1^\diamond, \dots, s_{n-1}^\diamond\}$.
\end{proposition}

\begin{remark}
If $\mu^-(w)=()$  then we interpret the given formula to mean $F_{K,\diamond}(q) = [n]_q!$.
\end{remark}

\begin{proof}
Write $\lambda(w) = a=(a_1,a_2,\dots,a_n)$ 
and $\beta^-(w) = (b_1,b_2,\dots,b_{m})$ and note that  
\[ \beta(w)=\beta(\lambda(w)) = (b_1,\ b_2,\ \dots,\ b_{m-1},\ b_m,\ n-b_m,\ n-b_{m-1},\ \dots,\ n-b_2,\ n-b_1,\ n)\]
since $\lambda(w)$ is antisymmetric.  
Set $b_0 = 0$ and for each $i \in [m]$ define 
$
A_i = \{ s_{k}: b_{i-1} < k < b_i\}
$ and 
$
B_i = \{ s_{k} : n-b_i < k < n-b_{i-1}\},
$
and also let
$F = \{ s_k : b_m < k < n-b_m\}$. Note that these $2m+1$ subsets of simple generators are pairwise disjoint, and that some of the sets may be empty; in particular, we have $|A_i| = |B_{m+1-i}| = c_i-1$ and $|F| = z-1$.

By Theorem \ref{Omega-thm}
we have $w=w^{-1}=w_a$, and the following properties derive easily from the definition of $w_a$ in Lemma \ref{a-lem}:
\begin{itemize}
\item[(i)] For each $i \in [m]$, $\diamond$ restricts to inverse bijections $A_i \to B_{m+1-i}$ and $B_i \to A_{m+1-i}$
 which are  order-preserving (with the set of simple generators  ordered by index in the obvious way).
 
 \item[(ii)] The automorphism $\diamond$ fixes every element of $F$ since $w(i)=i$ for all $b_m<i\leq n-b_m$.
 
 \item[(iii)] Let $i \in [n-1]$. If $i$ is a term in the sequence $\beta(w)$, or equivalently if $s_i$ does not belong to any of the  sets $A_1,A_2,\dots,A_m$ or $B_1,B_2,\dots,B_m$ or $F$, then $s_i^\diamond \notin \{ s_1,s_2,\dots,s_{n-1}\}$.
 
 \end{itemize}
 We conclude from these facts that $K = \bigcup_{i \in [m]} A_i \cup B_i \cup F$. Moreover, writing $W = \tilde S_n$, it follows that we may identify the standard parabolic subgroup $W_K\subset \tilde S_n$ with the cartesian product
 \[W_K = S_{c_1} \times S_{c_2} \times \dots \times S_{c_m} \times S_z \times S_{c_m} \times \dots \times  S_{c_2} \times S_{c_1}\]
 and that $\diamond$ acts on $W_K$ with respect to this identification by the formula
 \[ (w_1,w_2,\dots,w_m,x,w_m',\dots,w_2',w_1')^\diamond  = (w_1',w_2',\dots,w_m',x,w_m,\dots,w_2,w_1)\]
 for $w_i ,w_i'\in S_{c_i}$ and $x \in S_z$.
In this sense, the elements in $W_K$ fixed by $\diamond$ are precisely the tuples of the form
$(w_1,w_2,\dots,w_m,x,w_m,\dots,w_2,w_1)$ where $w_i \in S_{c_i}$ and $x \in S_z$.
 As the length of such a generic fixed element is $\ell(x) + 2\sum_{i=1}^m \ell(w_i)$, we deduce that 
$ F_{K,\diamond}(q) = \prod_{i=1}^m P_{S_{c_i}}(q^2) \cdot P_{S_z}(q) $ which coincides with the desired formula as $P_{S_k}(q) = [k]_q!$ by Theorem \ref{Pfactor-thm}.
\end{proof}

Combining the preceding results yields the following sum-to-product identity, which will be the first main step in our proof of Theorem \ref{main-thm} from the introduction.

 \begin{lemma}\label{mufixed-lem}
 Fix  $m \in \NN$ and let $c=(c_1,c_2,\dots,c_m)$ be a sequence of positive integers with $\sum_{i=1}^m c_i \leq \lfloor \frac{n}{2}\rfloor$.
 Define $b_k = \sum_{i=1}^k c_i$  and $ z = n - 2(c_1+c_2+\dots+c_m)$.
 It then holds that
 \[ 
\sum_{\substack{(w,\diamond,K) 
\\ \mu^-(w) =c}}
  q^{\ell(w)} \cdot \tfrac{ P_{S_n}(q^2)}{F_{K,\diamond}(q)}
 =
q^{\binom{z}{2}-\binom{n}{2}} \cdot \tfrac{[n]_{q^2}!}{[z]_q!}\cdot
 \prod_{i=1}^m
 \(  \frac{q^{2b_i(n-b_i)}}{1-q^{2b_i(n-b_i)}
}
\cdot
\frac{q^{2\binom{c_i}{2}}}{[c_i]_{q^2}!}
\)
 \]
where the sum is over the triples $(w,\diamond,K)$ with $w \in \Imin_n$ such that $\mu^-(w) = c$, and with $\diamond$ and $K$  defined relative to $w$ as in Proposition \ref{FOmega-prop}.
 \end{lemma}

 \begin{proof}
%Abbreviate by setting $c=(c_1,c_2,\dots,c_m)$.
 Recall that $P_{S_n}(q^2) = [n]_{q^2}!$ by Theorem \ref{Pfactor-thm} (since the degrees of $S_n$ are $2,3,\dots,n$).
Given this, it follows from Propositions \ref{ellOmega-prop} and \ref{FOmega-prop}
that 
\[
\sum_{\substack{(w,\diamond,K) \\ \mu^-(w) = c}}
  q^{\ell(w)} \cdot \tfrac{ P_{S_n}(q^2)}{F_{K,\diamond}(q)}
  =
  q^{\binom{z}{2}-\binom{n}{2}} \cdot \tfrac{[n]_{q^2}!}{[z]_q!}\cdot
 \prod_{i=1}^m
\frac{q^{2\binom{c_i}{2}}}{[c_i]_{q^2}!}
\cdot
\sum_{\substack{w \in \Imin_n \\ \mu^-(w) = c}}
 \prod_{i=1}^m q^{2b_i(n-b_i) d_i(w) }
\]  
where in the last product on the right we define $d_i(w)$ to  be the $i^{\mathrm{th}}$ element of the sequence $\delta^-(w)$.
In view of Lemma \ref{delta-lem}, it holds that
\[
\sum_{\substack{w \in \Imin_n \\ \mu^-(w) = c}}
 \prod_{i=1}^m q^{2b_i(n-b_i) d_i(w) }
=
 \prod_{i=1}^m \sum_{d =1}^\infty q^{2b_i(n-b_i) d } = \prod_{i=1}^m \frac{q^{2b_i(n-b_i)}}{1-q^{2b_i(n-b_i)}}.
 \]
Substituting the last expression  into  the preceding equation gives the desired formula.
 \end{proof}

 \def\Seq{\Sigma}
 We wish to give a closed formula for the sum of the right side of Lemma \ref{mufixed-lem} over all positive integer sequences $c$ with a fixed sum. It turns out that we may accomplish this by invoking a more general identity, which we prove next.
 Let $x$ and $y$ be indeterminates. 
 Given a sequence of positive integers $c=(c_1,c_2,\dots,c_m)$, define $b_i =(c_1+c_2+\dots+c_i)$ for $i \in [m]$ and set
 \be\label{Pi-eq} \Pi(c;x,y) = \prod_{i=1}^m \( \frac{1}{x^{(b_i)^2} y^{b_i}-1} \cdot \frac{x^{\binom{c_i}{2}}}{[c_i]_x!}\).\ee
 When $c=()$ is the empty sequence, then we set $\Pi(c;x,y) = 1$, following our usual conventions governing empty products. Observe that in this notation, the product on the right side of the equation in Lemma \ref{mufixed-lem} is $  \Pi(c;q^2,q^{-2n}) $.
 
 \begin{proposition}\label{csumfixed-prop} Fix a nonnegative integer $k$ and let $C_{k}$ be the set of all positive integer  sequences $c=(c_1,c_2,\dots,c_m)$, of any length $m \geq 0$,
 such that $c_1+c_2+\dots + c_m=k$. Then
    \[
\sum_{c \in C_k } \Pi(c;x,y) = 
\frac{1}{[k]_x!} \prod_{i=1}^k \frac{1}{x^iy-1} .
% \frac{[n]_x}{[n-k]_x} \cdot x^{k(2n-k-1)/2} \cdot \tbinom{n+k}{k}_x \cdot \frac{[n-k]_x!}{[n+k]_x!} \cdot (1-x)^{-(n+2k)} 
 \] 
 \end{proposition}
 
 \begin{remark}
  We consider $C_0$ to be the set with one element given by the unique empty sequence.
 \end{remark}
 
 Before giving the proof of Proposition \ref{csumfixed-prop}, we require a brief technical lemma.
 
 \begin{lemma}\label{tech-lem}
If $k \in \NN$ and $x,y$ are indeterminates then $\sum_{i=0}^k  { k \choose i}_x \cdot (y;x)_{k-i} \cdot y^{i} =1$.
\end{lemma}

\begin{proof}
Define $F_k(x,y) =\sum_{i=0}^k  { k \choose i}_x \cdot (y;x)_{k-i} \cdot y^{i}$.
For $0\leq i \leq k$ it holds that
  $ (1-y)(xy;x)_{k-i} = (y;x)_{k+1-i}$ and $ x^i \tbinom{k}{i}_x + \tbinom{k}{i-1}_x = \tbinom{k+1}{i}_x$.
  Using these identities, it is straightforward to check that
$ F_{k+1}(x,y) =  (1-y) F_k(x,xy) + y F_k(x,y)$, from which 
 it follows   by induction that $F_k(x,y) =1$.
\end{proof}

 \begin{proof}[Proof of Proposition \ref{csumfixed-prop}]
The desired identity  holds trivially when $k=0$. Assume $k>0$ and 
let $C^{(i)}_k$ denote the subset of sequences in $ C_k$ whose last term is $i$.
 Then 
 \[ 
 \sum_{c \in C_k } \Pi(c;x,y) 
 =
 \sum_{i=1}^k \sum_{c \in C_k^{(i)}} \Pi(c;x,y) 
 =
 \frac{1}{x^{k^2}y^k-1}   \sum_{i=1}^k  \( \frac{x^{\binom{i}{2}}}{[i]_x!}  \sum_{c \in C_{k-i}} \Pi(c;x,y) \)
 .\] 
When $i \in [k]$ we may assume by induction that $\sum_{c \in C_{k-i}} \Pi(c;x,y) = \frac{1}{[k-i]_x!} \prod_{i=1}^{k-i} \frac{1}{x^iy-1}$.
 After making this substitution  in the preceding equation and carrying out a few straightforward manipulations, we deduce that 
 \[
   \sum_{c \in C_k } \Pi(c;x,y) =  \(\frac{1}{[k]_x!}   \prod_{i=1}^k \frac{1}{x^iy-1}\) \frac{\Sigma}{z^k-1}
 \]
 where 
$z = x^ky$ and $ \Sigma = \sum_{i=1}^k      \tbinom{k}{i}_x \cdot  (z^{-1};x)_i  \cdot z^i.$
It   remains just  to show  that $\Sigma = z^k-1$,
and 
 for this  we  note that 
 $ \Sigma =
z^k \(\sum_{i=0}^k  \tbinom{ k }{ i}_x  \cdot (z^{-1};x)_{k-i} \cdot z^{-i}\)- 1 
 $
and then invoke Lemma \ref{tech-lem}.
 \end{proof}

 \subsection{Proof of main theorem}\label{last-sect}

Let $(W,S)$ be a finite Weyl group and write $(\tilde W,\tilde S)$ for its associated affine Weyl group, and 
 for each nonnegative integer $k$ define %we write $I_{\tilde W}^{(k)}(q)$ by
 \[
 \Sigma_{\tilde W}(k,q) = 
  \sum_{\substack{(w,\diamond,K) \\ \ellprime (w) = k}}
  q^{\ell(w)} \cdot \tfrac{ P_{W}(q^2)}{F_{K,\diamond}(q)} \in \ZZ[[q]]
\]
where the sum is over  the $(\tilde W, S,\id)$-double coset data $(w,\diamond,K)$ with $\ellprime (w)=\ellinv{\id}(w)=k$.
Observe from \eqref{Tdef} that $ T_{\tilde W}(s,q) =
 \tfrac{P_{\tilde W}(q)}{P_{\tilde W}(q^2)}\cdot  \sum_{ k  \geq 0} \(s\cdot \tfrac{1-q}{1+q}\)^k \cdot 
\Sigma_{\tilde W}(k,q).
$
We  prove the following closed formula for $\Sigma_{\tilde W}(k,q)$ when $\tilde W =\tilde S_n$.
This result  is equivalent to Theorem \ref{main-thm}
%since $ {P_{\tilde S_n}(q)}/{P_{\tilde S_n}(q^2)} = \tfrac{(1+q)^n}{1+q^n}$ 
since Theorem \ref{bott-thm} implies $ {P_{\tilde S_n}(q)}/{P_{\tilde S_n}(q^2)} = \tfrac{(1+q)^n}{1+q^n}$ and   Proposition \ref{ell*-prop} implies that $\Sigma_{\tilde S_n}(k,q)=0$ if $2k>n$.

 \begin{theorem}\label{last-thm}
 For each positive integer $n$ and nonnegative integer $k \leq \lfloor \frac{n}{2}\rfloor$,
 it holds that
 \[
\Sigma_{\tilde S_n}(k,q)
  = 
  q^{k^2} \cdot \tfrac{[n]_q}{[2n-2k]_q} \cdot \tbinom{n-k}{k}_q \cdot (-q^{k+1};q)_{n-2k} \cdot \(\tfrac{1+q}{1-q} \)^k \cdot \tfrac{1+q^n}{(1+q)^n}
  .\]
  \end{theorem}

\begin{proof}
On noting the formula for $\ellprime (w)$  in Proposition \ref{ellOmega-prop}, we deduce from Lemma \ref{mufixed-lem}
that 
\be\label{Ipenum-eq}
\Sigma_{\tilde S_n}(k,q)
  =
 q^{\binom{n-2k}{2}-\binom{n}{2}} \cdot \tfrac{[n]_{q^2}!}{[n-2k]_q!}\cdot \sum_{c \in C_k} \Pi(c;q^2,q^{-2n})
 \ee
 with $C_k$ as in Proposition \ref{csumfixed-prop} and $\Pi(c;x,y)$  as in \eqref{Pi-eq}.
We now observe that the following identities  hold (for all integers $0\leq k \leq \lfloor \tfrac{n}{2}\rfloor$):
\ben
\item[(i)] $\sum_{c \in C_k} \Pi(c;q^2,q^{-2n}) = \tfrac{1}{ [k]_{q^2}!} \cdot \prod_{i=1}^k \frac{q^{2(n-i)}}{1-q^{2(n-i)}}.$

\item[(ii)] $q^{\binom{n-2k}{2}-\binom{n}{2}} \prod _{i=1}^k q^{2(n-i)} = q^{k^2}$.

\item[(iii)] $ [n]_{q^2}! \cdot  \prod_{i=1}^k \frac{1}{1-q^{2(n-i)}}
=
\frac{[n]_q}{[2n-2k]_q} \cdot  [n-k]_{q^2}! \cdot \frac{1+q^n}{(1-q^2)^k}
$.

\item[(iv)] $ \frac{ [n-k]_{q^2}!}{[n-2k]_q!} \cdot \frac{1}{ [k]_{q^2}!}
=
 \binom{n-k}{k}_q \cdot  (-q^{k+1};q)_{n-2k} \cdot (1+q)^{2k-n}$.
\een
Only the first of these equations (which follows from Proposition \ref{csumfixed-prop}) is nontrivial;
the rest derive straightforwardly from the standard definitions of the various $q$-numbers 
 given before Theorem \ref{main-thm}.
Applying these substitutions to \eqref{Ipenum-eq} produces the desired formula.
 \end{proof}

\end{document}